\documentclass{amsproc}

\newtheorem{theorem}{Theorem}[section]
\newtheorem{lemma}[theorem]{Lemma}
\newtheorem{proposition}[theorem]{Proposition}
\newtheorem{corollary}[theorem]{Corollary}
\theoremstyle{definition}
\newtheorem{definition}[theorem]{Definition}
\newtheorem{example}[theorem]{Example}

\theoremstyle{remark}
\newtheorem{remark}[theorem]{Remark}

\numberwithin{equation}{section}

\newcommand{\Z}{\mathbb{Z}}
\newcommand{\N}{\mathbb{N}}
\newcommand{\g}{\gcd}
\newcommand{\m}{\ \text{\rm mod }}
\newcommand{\eg}{\overset{*}{=}}

\newcommand{\uloopr}[1]{\ar@'{@+{[0,0]+(-4,5)}@+{[0,0]+(0,10)}@+{[0,0] +(4,5)}}^{#1}}
\newcommand{\uloopd}[1]{\ar@'{@+{[0,0]+(5,4)}@+{[0,0]+(10,0)}@+{[0,0]+ (5,-4)}}^{#1}}
\newcommand{\dloopr}[1]{\ar@'{@+{[0,0]+(-4,-5)}@+{[0,0]+(0,-10)}@+{[0, 0]+(4,-5)}}_{#1}}
\newcommand{\dloopd}[1]{\ar@'{@+{[0,0]+(-5,4)}@+{[0,0]+(-10,0)}@+{[0,0 ]+(-5,-4)}}_{#1}}
\newcommand{\luloop}[1]{\ar@'{@+{[0,0]+(-8,2)}@+{[0,0]+(-10,10)}@+{[0, 0]+(2,2)}}^{#1}}

\usepackage{amssymb}
\usepackage{amsfonts}
\usepackage{latexsym}
\usepackage{amscd}
\usepackage[mathscr]{euscript}
\usepackage{xy} \xyoption{all}

\begin{document}

\title{The Leavitt path algebras of generalized Cayley graphs }
\author{Gene Abrams}
\address{Department of Mathematics, University of Colorado,
Colorado Springs, CO 80918 U.S.A.}
\email{abrams@math.uccs.edu}
\thanks{The first  author is partially supported by a Simons Foundation Collaboration Grants for
Mathematicians Award \#208941. The second author was partially supported by the Spanish MEC and Fondos FEDER through project MTM2010-15223, and by the Junta de Andaluc\'{\i}a and Fondos FEDER, jointly, through projects FQM-336 and FQM-2467. Part of this work was carried out during a visit of the second author to the University of Colorado Colorado Springs. The second author thanks this host institution for its warm hospitality and support.}

\author{Gonzalo Aranda Pino}
\address{Department of Algebra, Geometry and Topology, University of M\'alaga, 29071, Spain}
\email{g.aranda@uma.es}




\begin{abstract}
Let $n$ be a positive integer.  For each $0\leq j \leq n-1$ we let $C_n^{j}$ denote Cayley graph for the cyclic group  $\Z_n $  with respect to the subset  $\{1, j\}$.  For any such pair $(n,j)$ we  compute the size of the Grothendieck group of  the Leavitt path algebra $L_K(C_n^j)$; the analysis is related to a collection of integer sequences described by Haselgrove in the 1940's.     When $j=0,1,$ or $2$, we are able to extract enough additional information about the structure of these Grothendieck groups  so that we may apply  a Kirchberg-Phillips-type  result to explicitly realize the algebras $L_K(C_n^j)$ as the Leavitt path algebras of graphs having at most three vertices.   The analysis in the $j=2$ case leads us to some perhaps surprising and apparently nontrivial connections to the classical Fibonacci sequence.  

\bigskip
\bigskip
\begin{center}
Dedicated to the memory of William G. ``Bill" Leavitt, 1916-2013.
\end{center}
\end{abstract}

\maketitle

Let $n$ be a fixed positive integer, and let $j$ be an integer for which $0 \leq j \leq n-1$.  We define the {\it Cayley graph} $C_n^j$ to be the directed graph consisting of $n$ vertices $\{v_1, v_2, \hdots, v_n\}$  and $2n$ edges $\{e_1, e_2, \hdots, e_n, f_1, f_2, \hdots, f_n\}$ for which 
$s(e_i) = v_i,  \  r(e_i) = v_{i+1},  \  s(f_i) = v_i,   $ and $ r(f_i) = v_{i+j},$
where subscripts are interpreted ${\rm mod } \ n$, and where $s(e )$ (resp., $r(e )$) denotes the source (resp., range) vertex of the edge $e$.    For clarity, we picture the three graphs $C_4^0$, $C_4^1$, and $C_4^2$ here.

 $$  C_4^0 = \ \ \ \ {
\def\labelstyle{\displaystyle}
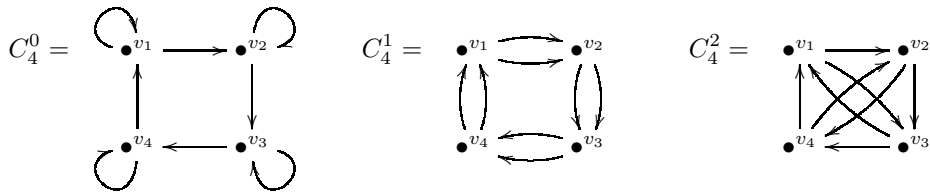
\xymatrix{ \bullet^{v_1} \ar[r]  \ar@(l,u) &  \bullet^{v_2} \ar@(u,r) \ar[d] \\  \bullet^{v_4} \ar@(d,l) \ar[u] & \bullet^{v_3} \ar[l] \ar@(r,d)}} \hskip1cm  \ C_4^1 =  \ \ {
\def\labelstyle{\displaystyle}
\xymatrix{ \bullet^{v_1} \ar@/^{5pt}/ [r] \ar@/_{5pt}/[r]   &  \bullet^{v_2} \ar@/^{5pt}/ [d] \ar@/_{5pt}/[d] \\  \bullet^{v_4}  \ar@/^{5pt}/ [u] \ar@/_{5pt}/[u] & \bullet^{v_3} \ar@/^{5pt}/ [l] \ar@/_{5pt}/[l]
}} \hskip1cm  C_4^2 = \ \   {
\def\labelstyle{\displaystyle}
\xymatrix{ \bullet^{v_1} \ar@/^{5pt}/ [dr] \ar[r]   &  \bullet^{v_2} \ar@/^{5pt}/ [dl] \ar[d] \\  \bullet^{v_4}  \ar@/^{5pt}/ [ur] \ar[u] & \bullet^{v_3} \ar@/^{5pt}/ [ul] \ar[l]
}}
$$

\vskip0.5cm

\noindent
  Less formally, $C_n^j$ is the graph with $n$ vertices and $2n$ edges, where each vertex $v$ emits two edges, one to its ``next" neighboring vertex (which we will draw throughout this article in a clockwise direction), and one to the vertex $j$ places clockwise from $v$.    More formally,  for any finite group $G$, and any set of generators $S$ of $G$, the {\it Cayley graph of } $G$ {\it with respect to} $S$ is the directed graph $E_{G,S}$ having vertex set $\{v_g \ | \ g\in G\}$, and in which there is an edge from $v_g$ to $v_h$ in case there exists $s\in S$ with $h = gs$ in $G$.  Thus,  for $n\geq 3$, the graph $C_n^j$ is the Cayley graph $E_{G,S}$ of the cyclic group $G = \Z_n$ with respect to the generating subset $S = \{1,j\}$ of $\Z_n$. 
  
  In \cite{ASch} a description is given of the  Leavitt path algebras of the form  $L_K(C_n^{n-1})$. Specifically, it is shown in \cite[Theorem 8]{ASch}, that there are exactly four isomorphism classes represented by the collection $\{ L_K(C_n^{n-1}) \ | \ n\in \mathbb{N}\}$. 
  
  In the current article we continue the investigation of Leavitt path algebras associated to Cayley graphs.  In Section \ref{Background} we give the germane definitions, and present background information about our main tool, the Algebraic KP Theorem.  In Section \ref{Haselgrovenumbers} we compute the two important (and closely related) integers $|K_0(L_K(C_n^j))|$ and ${\rm det}(I - A_{C_n^j}^t)$, where $K_0(  -  )$ denotes the standard Grothendieck group of a ring, and $A_{ ( -  ) }$ denotes the incidence matrix of a graph.  Specifically, we show that, for fixed $j$, the integers $|K_0(L_K(C_n^j))|$ are (up to sign)  the entries in the ``$j^{{\rm th}}$  Haselgrove sequence", a sequence investigated by Haselgrove \cite{H} as part of a possible approach to  establishing Fermat's Last Theorem.     In Section \ref{SectionC^0andC^1} we completely describe up to isomorphism the collection $\{ L_K(C_n^0) \ | \ n\in \mathbb{N}\}$ (there is only one such algebra), and the collection $\{ L_K(C_n^1) \ | \ n\in \mathbb{N}\}$ (there are infinitely many such algebras);  each of these algebras is subsequently realized in terms of (matrices over) the standard {\it Leavitt algebras} $L_K(1,m)$.   We conclude with Section \ref{SectionC^2}, in which we completely describe up to isomorphism the collection $\{ L_K(C_n^2) \ | \ n\in \mathbb{N}\}$.  As we shall see, there are infinitely many such algebras in this class.  We are able to  describe the groups $K_0(L_K(C_n^2))$ explicitly in terms of the integers appearing in the second Haselgrove sequence, together with integers arising from the standard Fibonacci sequence.   The descriptions of the algebras presented in Sections \ref{SectionC^0andC^1} and \ref{SectionC^2} will follow from an application of the Algebraic KP Theorem.

  \section{Background information: \\ Leavitt path algebras, and the Algebraic KP Theorem}\label{Background}
  
For any field $K$ and directed graph $E$ the Leavitt path algebra $L_K(E)$ has been the focus of sustained investigation since 2004.  We give here a basic description of $L_K(E)$; for additional information, see e.g.   \cite{AAP1} or \cite{TheBook}.  

\medskip

{\bf Definition of Leavitt path algebra.}   Let $K$ be a field, and let $E = (E^0, E^1, r,s)$ be a directed  graph with vertex set $E^0$ and edge set $E^1$.   The {\em Leavitt path $K$-algebra} $L_K(E)$ {\em of $E$ with coefficients in $K$} is  the $K$-algebra generated by a set $\{v\mid v\in E^0\}$, together with a set of variables $\{e,e^*\mid e\in E^1\}$, which satisfy the following relations:

(V)   \ \ \  \ $vw = \delta_{v,w}v$ for all $v,w\in E^0$, \  

  (E1) \ \ \ $s(e)e=er(e)=e$ for all $e\in E^1$,

(E2) \ \ \ $r(e)e^*=e^*s(e)=e^*$ for all $e\in E^1$,

 (CK1) \ $e^*e'=\delta _{e,e'}r(e)$ for all $e,e'\in E^1$,

(CK2)Ê\ \ $v=\sum _{\{ e\in E^1\mid s(e)=v \}}ee^*$ for every   $v\in E^0$ for which $0 < |s^{-1}(v)| < \infty$.

An alternate description of $L_K(E)$ may be given as follows.  For any graph $E$ let $\widehat{E}$ denote the ``double graph" of $E$, gotten by adding to $E$ an edge $e^*$ in a reversed direction for each edge $e\in E^1$.   Then $L_K(E)$ is the usual path algebra $K\widehat{E}$, modulo the ideal generated by the relations (CK1) and (CK2).     \hfill $\Box$

\medskip

It is easy to show that $L_K(E)$ is unital if and only if $|E^0|$ is finite.  This is of course the case when $E = C_n^j$.  

For a directed graph $E$ having $n$ vertices $v_1, v_2,Ê\dots, v_n$  we denote by $A_E$ the usual {\it incidence matrix} of $E$, namely, the matrix $A_E = (a_{i,j})$ where, for each pair $1\leq i,j \leq n$, the entry $a_{i,j}$ denotes the number of edges $e$ in $E$ for which $s(e)=v_i$ and $r(e) = v_j$.

Let $E$ be a directed graph  with vertices $v_1, v_2, \hdots, v_n$ and incidence matrix $A_E = (a_{i,j})$.  We let $F_n$ denote the free abelian monoid on the generators $v_1, v_2, \hdots, v_n$ (so $F_n \cong \bigoplus_{i=1}^n \Z^+$ as monoids).  We denote the identity element of this monoid by $z$.    We define a relation $\approx$ on $F_n$ by setting 
$$v_i  \approx  \sum_{j=1}^n a_{i,j}v_j$$
for each non-sink $v_i$, and  denote by $\sim_E$ the equivalence relation on $F_n$ generated by $\approx$.        
For two $\sim_E$ equivalence classes $[x]$ and $[y]$ we define $[x] + [y] = [x+y]$; it is straightforward to show that this gives a well-defined associative binary operation on the set of $\sim_E$ equivalence classes, and that $[z]$ acts as an identity element for  this operation.  We denote the resulting 
 {\it graph monoid}  $F_n / \sim_E$  by $M_E.$

 \medskip
 
\begin{definition} For any $n\geq 1$ and $0\leq j \leq n-1$, the graph monoid  $M_{C_n^j}$ is the monoid generated by $[v_1], [v_2], \hdots, [v_n]$, subject to the relations $$[v_i ]= [v_{i+1}] + [v_{i+j}]$$ 
(for all $1\leq i \leq n$), where subscripts are interpreted ${\rm mod} \ n$.  
\end{definition}

\begin{example}\label{monoidC42example}
As a representative example, we explicitly describe the graph monoid $M_{C_4^2}$ associated to the graph $C_4^2$.   The generators of this monoid are $[v_1],[v_2],[v_3],$ and $[v_4]$, subject to the relations
$$[v_1]=[v_2]+[v_3], \ [v_2]=[v_3]+[v_4], \ [v_3]=[v_4]+[v_1], \ \mbox{and} \ [v_4]=[v_1]+[v_2].$$
Consider the element $\sigma =[v_1]+[v_2]+[v_3]+[v_4]$ of  $M_{C_4^2}$.  Then we get 
$$[v_1] + \sigma = [v_1] +  [v_1]+[v_2]+ ([v_3]+[v_4]) = [v_1] +  [v_1]+[v_2]+ [v_2] = [v_1] +  [v_2]+[v_1]+ [v_2] = 2[v_4].$$
But upon reconsidering $[v_1]+\sigma$ we also get
$$[v_1] + \sigma = [v_1] +  [v_1]+[v_2]+ [v_3]+[v_4] =( [v_1] +  [v_4])+([v_1] + [v_2])+ [v_3] = [v_3] + [v_4] + [v_3] = [v_2]+[v_3] = [v_1].$$ 
Thus $[v_1] = [v_1] + \sigma = 2[v_4]$.   Arguing similarly on the elements $[v_i]+\sigma$ for $i = 2,3,4$, we get $[v_2] = 2[v_1], [v_3]=2[v_2],$ and $[v_4]=2[v_3]$.   But then $[v_3] = 4[v_1]$ and $[v_4]=8[v_1]$, so each of the generators of $M_{C_4^2}$  is seen to be a positive integer multiple of $[v_1]$.  Thus each element of $M_{C_4^2}$ is of the form $j[v_1]$ for some integer $j\geq 0$.  Now $$6[v_1]  = 2 [v_1]  + 4[v_1] =  2[v_1] + [v_3] = [v_2] + [v_3] = [v_1],$$
so that each element of  $M_{C_4^2}$ is of the form $[z], [v_1], 2[v_1], 3[v_1], 4[v_1], $ or $ 5[v_1]$.    Using an appropriate invariant, one can now show that these six elements are indeed distinct in $M_{C_4^2}$; however, we will establish this fact more efficiently in Remark \ref{sizeofMC42} below,  using different techniques.   (Because none of the defining relations for $\approx$ involve the element $z$, it follows immediately that $i [v_1] \neq [z]$ for $1\leq i \leq 5.$)      \hfill $\Box$
\end{example}

For an example similar to that given in Example \ref{monoidC42example},   see also \cite[p. 3]{ASch}, in which the graph monoid  $M_{C_3^2}$ is explicitly shown to be the five elements $\{[z], \ [v_1], \  [v_2], \ [v_3], \ [v_1]+[v_2]+[v_3] \}.$

\smallskip

  We present now a streamlined version of the germane background information which will be utilized throughout the remainder of the article.  Much of this discussion appears in \cite{ASch}.  

For a unital $K$-algebra $A$,   the set of isomorphism classes of finitely generated projective left $A$-modules is denoted by $\mathcal{V}(A)$.  We denote the elements of  $\mathcal{V}(A)$ using brackets; for example, $[A] \in \mathcal{V}(A)$ represents the isomorphism class of the left regular module ${}_AA$.   $\mathcal{V}(A)$ is a monoid, with operation $\oplus$, and zero element $[\{0\}]$.   The monoid $(\mathcal{V}(A), \oplus)$ is {\it conical}; this means that the sum of any two nonzero elements of $\mathcal{V}(A)$ is nonzero, or, rephrased, that $\mathcal{V}(A)^* = \mathcal{V}(A) \setminus \{0\}$ is a semigroup under $\oplus$. 
A striking property of Leavitt path algebras  was established in \cite[Theorem 3.5]{AMP}, to wit:  
\begin{equation}
  \mathcal{V}(L_K(E)) \cong M_E \ \mbox{as monoids}.    \tag{$*$}
 \end{equation}
\vspace{-.2in}
$$\mbox{Moreover, } \   [L_K(E)] \leftrightarrow \sum_{v\in E^0} [v]  \ \mbox{under this isomorphism.}$$

\medskip

A unital $K$-algebra $A$ is called {\it purely infinite simple} in case $A$ is not a division ring, and $A$ has the property that for every nonzero element $x$ of  $A$ there exist $b,c\in A$ for which $bxc=1_A$.     It is shown in \cite[Corollary 2.2]{AGP} that if $A$ is  a unital purely infinite simple $K$-algebra, then the semigroup $(\mathcal{V}(A)^*, \oplus)$ is in fact a group, and, moreover, that $\mathcal{V}(A)^* \cong K_0(A)$, the Grothendieck group of $A$.     
   Summarizing: when $L_K(E)$ is unital purely infinite simple we have the following isomorphisms of groups:  
\begin{equation}
 K_0(L_K(E)) \cong \mathcal{V}(L_K(E))^* \cong M_E^*. \tag{$**$}
 \end{equation}
In particular, in this situation we have   $|K_0(L_K(E))| = | M_E^* |$.   Often throughout the article we will invoke this isomorphism (and the resultant equality of set sizes)  without explicit mention.   (Notationally, the elements of $M_E^*$ are somewhat easier to work with than are those of $K_0(L_K(E))$; on the other hand, many of the results in the literature related to this topic are stated in terms of $K_0(L_K(E))$.  Thus  we will typically do computations in the group $M_E^*$, and use those to immediately draw the appropriate conclusions in  $K_0(L_K(E))$.)  
 
 \medskip

The finite graphs $E$ for which the Leavitt path algebra $L_K(E)$ is purely infinite simple have been explicitly described in \cite{AAP2}, to wit: $L_K(E)$ is purely infinite simple if and only if  $E$ is cofinal, sink-free, and satisfies Condition (L).    (See e.g. \cite{AAP2} for a complete description of these terms.)   This result, together with the preceding discussion, immediately yields
\begin{proposition}\label{LKCnpis}
For each $n\geq 1$, and for each $0\leq j \leq n-1$,  the $K$-algebra $L_K(C_n^j)$ is unital purely infinite simple.  In particular, $M_{C_n^j}^* = (M_{C_n^j} \setminus \{[z]\},+)$ is a group. 
\end{proposition}

\begin{example}\label{groupofmonoidC42example}
We revisit the monoid $M_{C_4^2}$  described in Example \ref{monoidC42example}, in which we showed that every element of  $M_{C_4^2}$ is of the form $[z], [v_1], 2[v_1], 3[v_1], 4[v_1], $ or $ 5[v_1]$.  This yields that $M_{C_4^2}^* = \{ [v_1], 2[v_1], 3[v_1], 4[v_1],  5[v_1]\} \cong \Z_5$, and that $5[v_1]$ is the identity element of  $M_{C_4^2}^*$.  

We recall that it was established in Example \ref{monoidC42example} that $[v_1] + \sigma = [v_1]$, where $\sigma$ is the element $[v_1]+[v_2]+[v_3]+[v_4]$ of $M_{C_4^2}^*$.   But since we have now established that  $M_{C_4^2}^*$ is in fact a group, this equation implies that $\sigma$ is the identity element of   $M_{C_4^2}^*$. \hfill $\Box$ 
\end{example}

 Additionally, referring to the explicit description of $M_{C_3^2}$ given in \cite{ASch}, we see that $M_{C_3^2}^* = \{ [v_1], \  [v_2], \ [v_3], \ [v_1]+[v_2]+[v_3] \} \cong \Z_2 \times \Z_2,$ and that $ [v_1]+[v_2]+[v_3 ]$ is the identity element of the group $M_{C_3^2}^*$.  

The observations made in the previous two paragraphs regarding the element $\sum_{i=1}^n [v_i]$ of $M_{C_n^j}^*$ in fact holds in general.

\begin{proposition}
\label{sumverticesisidentity}
For each $n\in \mathbb{N}$ and $0\leq j \leq n-1$, $\sum_{i=1}^n [v_i ]$ is the identity element of the group  $M_{C_n^j}^*$.
\end{proposition}

{\bf Proof}:   Let $\sigma$ denote  $\sum_{i=1}^n [v_i ]$.  Using the defining relations in $M_{C_n^j}^*$, we have  
$$ \sigma = \sum_{i=1}^n [v_i ] = \sum_{i=1}^n ([v_{i+1} ] + [v_{i+j} ]) = \sum_{i=1}^n [v_{i+1} ] + \sum_{i=1}^n [v_{i+j} ] = \sigma + \sigma,$$
as clearly both of the final two summation expressions also give the sum over all vertices of the germane equivalence classes.       But the equation $\sigma +\sigma =\sigma$ in a group yields immediately that $\sigma$ is the identity element.   \hfill $\Box$

\medskip

We now have the necessary background information in hand which allows us to present the powerful tool which will yield a number of key results in Sections \ref{SectionC^0andC^1} and \ref{SectionC^2}.    

\medskip

 {\bf The Algebraic KP Theorem.} \cite[Corollary 2.7]{ALPS}  
Suppose $E$ and $F$ are finite graphs for which the Leavitt path algebras $L_K(E)$ and $L_K(F)$ are purely infinite simple.   Suppose that there is an isomorphism $\varphi : K_0(L_K(E)) \rightarrow K_0(L_K(F))$ for which $\varphi([L_K(E)]) = [L_K(F)]$, and suppose also that the two integers ${\rm det}(I_{|E^0|}~-~A_E^t)$ and ${\rm det}(I_{|F^0|}~-~A_F^t)$ have the same sign (i.e., are either both nonnegative, or  both nonpositive).    Then $L_K(E) \cong L_K(F)$ as $K$-algebras.

\smallskip

\begin{remark} The proof of the Algebraic KP Theorem utilizes deep results and ideas in the theory of symbolic dynamics. The letters K and P in its name derive from E. Kirchberg and N.C. Phillips, who (independently in 2000) proved an analogous result for graph C$^*$-algebras.  We note that, as of Fall 2013, it is not known whether the hypothesis regarding the germane  determinants can be eliminated from the statement of The Algebraic KP Theorem.   
\end{remark}

\section{The integers $|K_0(L_K(C_n^j))|$ and ${\rm det}(I-A_{C_n^j}^t)$}\label{Haselgrovenumbers}

Let $E$ be a finite directed graph for which $|E^0| = n$, and let $A_E$ denote the usual incidence matrix of $E$.   
We view $I_n - A_E^t$ both as a matrix, and as a linear transformation $I_n - A_E^t: \Z^n \rightarrow \Z^n$,  via left multiplication (viewing elements of $\Z^n$ as column vectors).   
In the situation where $L_K(E)$ is purely infinite simple, so that in particular $M_E^*$ is a group (necessarily isomorphic to $K_0(L_K(E))$), we have that  $$K_0(L_K(E)) \cong M_E^* \cong \Z^n / {\rm Im }(I_n - A_E^t) = {\rm Coker}(I_n - A_E^t).$$
 Under this isomorphism   $[v_i ] \mapsto \vec{b_i} + {\rm Im }(I_n - A_E^t)$, where  $\vec{b_i}$ is the element of $\Z^n$ which is $1$ in the $i^{th}$ coordinate and $0$ elsewhere.  
Moreover, if $K_0(L_K(E))$ is finite, then an analysis of the Smith normal form of the matrix $I_n - A_E^t$ yields
\begin{equation}
|  K_0(L_K(E)) | = | {\rm det}(I_n - A_E^t) |.  \tag{$***$}
 \end{equation}
Conversely, $K_0(L_K(E))$ is infinite if and only if ${\rm det}(I_n - A_E^t) = 0.$  (See \cite[Section 3]{AALP} for a complete discussion.) 

\smallskip

  An $n\times n$ matrix $B = (b_{i,j})$ is {\it circulant} in case  $b_{i+1, j+1} = b_{i,j}$ for  $1\leq i \leq n-1$, $1\leq j \leq n-1$;  $b_{1,j+1} = b_{n,j}$ for $j+1 \leq n$; $b_{i+1,1} = b_{i,n}$ for $i+1\leq n$; and $b_{1,1}=b_{n,n}$.   That is, $B$ is circulant in case each subsequent row of $B$ is obtained from the previous row by moving each entry of the previous row one place to the right, and moving the last entry of the previous row to the first position of the subsequent row.   (The last row gets moved to the first row in this way as well.)   If $B$ is circulant, then 
$$ \det(B) = \prod\limits_{\ell=0}^{n-1}(b_1+b_{2}\omega_\ell+b_{3}\omega_\ell^2+\dotsb + b_n\omega_\ell^{n-1})$$ where  $(b_1 \ b_2 \ b_3 \ \cdots \ b_n)$ is the first row of $B$, and $\omega_\ell =  e^{\frac{2\pi i \ell }{n}}$ (for $1\leq \ell \leq n$) are the $n$ distinct  $n^{th}$ roots of unity in $\mathbb{C}$.    (See e.g. \cite{circulantref}.) 

In the case of the  Cayley graph $C_n^j$ (for $n\geq 3$ and $2\leq j \leq n-1$),  the matrix $B =  I_n-A_{C_n^j}^t$ is easily seen to be circulant, with  $b_1=1$, $b_2=b_{1+j}=-1$, and $b_\ell=0$ for all  $3 \leq \ell \leq n$ with $\ell \neq 1+j$.  Using 
 the displayed equation,  
we get
\begin{proposition}\label{determinantasproduct} For all $n\in \mathbb{N}$ and all $0\leq j\leq n-1$ we have
$$\det(I_n-A_{C_n^j}^t) \  = \ \prod\limits_{\ell=0}^{n-1}(1-\omega_\ell-\omega_\ell^{j}).
$$
\end{proposition}
We note that it is easy to establish that the  equation given in Proposition \ref{determinantasproduct} also holds  in the cases $n=1$ and $n=2$, as well as in the cases   $j=0$ and $j=1$. 

The integer $|\prod\limits_{\ell=0}^{n-1}(1-\omega_\ell-\omega_\ell^{j}) \ |$ was investigated in \cite{H}, with an eye towards establishing a connection between a resolution of Fermat's Last Theorem (at the time, of course, Fermat's Last {\it Conjecture})  and some integers which share properties of the Mersenne numbers. 

     In that context,  the integers $a_k(n) = \prod\limits_{\ell=0}^{n-1}(\omega_\ell^{k} + \omega_\ell -1 )$ are introduced  in \cite[Formula (7)]{H}.  Clearly Proposition \ref{determinantasproduct} yields  $|a_k(n)| =|\det(I_n-A_{C_n^k}^t)|$.   For Haselgrove (\cite[p 21]{H}), `` ...   the sign of the numbers $a_k(n)$ is irrelevant to the subject."  (Indeed, the  numbers $a_k(n)$ may be positive, negative, or zero, depending on the values of $n$ and $k$.  For instance, the signs of the integers in the sequence $a_2(n)$ alternate each term.)    In stark contrast, we will be keenly interested in the sign of $\det(I_n-A_{C_n^k}^t)$.  

\begin{definition}\label{Haselgrovedefinition}
 For notational convenience, and with Haselgrove's work  in mind, we will often denote  the integer $| \det(I_n-A_{C_n^k}^t) |$ by  $H_k(n)$, and, for fixed $k$, will refer to the sequence 
 $$H_k(1), H_k(2), H_k(3), ...$$ as the {\it $k^{th}$ Haselgrove sequence}.  
\end{definition}

By  equation $(***)$, we immediately get 

\begin{proposition}\label{sizeofK0isH(n)}
Let $n\in \mathbb{N},$ and $0\leq j \leq n-1$. 

\smallskip

(1)  If $H_j(n)>0$, then
$ |K_0(L_K(C_n^j))| = H_j(n).$

\smallskip

(2)  If $H_j(n)=0$, then $K_0(L_K(C_n^j))$ is infinite.

\end{proposition}

\begin{remark}\label{sizeofMC42}
As a specific consequence of Proposition \ref{sizeofK0isH(n)}, we see that the group $M_{C_4^2}^* \cong K_0(L_K(C_4^2))$ of Example \ref{groupofmonoidC42example} contains $H_2(4)$ elements.  It is shown in \cite{H} (and is also demonstrated in Section \ref{SectionC^2}) that $H_2(4)=5$, thus verifying the assertion made at the end of Example \ref{monoidC42example}. \hfill $\Box$
\end{remark} 
\medskip

With the observation that $|a_j(n)| =|\det(I_n-A_{C_n^j}^t)|$ in mind, it turns out that, in contrast to the fact that the sign of   $a_j(n)$  depends on the specific values of $n$ and $j$, we obtain the following.  

\begin{proposition}\label{detisnegative}   For each $n\geq 1$ and $0\leq j \leq n-1$, 
$$\det(I_n-A_{C_n^j}^t) \leq 0.$$
In particular, $\det(I_n-A_{C_n^j}^t) = -H_j(n)$ for every such pair $n,j$. 
\end{proposition}
{\bf Proof.}   By Proposition \ref{determinantasproduct}  we get 
\begin{equation*}
\begin{split}
\det(I_n-A_{C_n^j}^t) 
& = \ \prod\limits_{\ell=0}^{n-1}(1-\omega_\ell-\omega_\ell^{j}) \ = \ 
 \prod\limits_{\ell=0}^{n-1}(1-e^{\frac{2\pi i\ell}{n}}-e^{\frac{2\pi i\ell j}{n}})\\
& = \  \prod\limits_{\ell=0}^{n-1}(1-\cos\frac{2\pi \ell}{n}- i\sin\frac{2\pi \ell}{n}-\cos\frac{2\pi \ell j}{n}-i\sin\frac{2\pi \ell j}{n}).\\
\end{split}
\end{equation*}
\noindent
For each $0\leq \ell \leq n-1$, write 
$$z_\ell = 1-\cos\frac{2\pi \ell}{n}- i\sin\frac{2\pi \ell}{n}-\cos\frac{2\pi \ell j}{n}-i\sin\frac{2\pi \ell j}{n} \ \  \mbox{ as }  \ \ z_\ell = x_\ell + i y_\ell,$$
 where $x_\ell =  1-\cos\frac{2\pi \ell}{n}-\cos\frac{2\pi \ell j}{n}$ and $y_\ell = - \sin\frac{2\pi \ell}{n} - \sin\frac{2\pi \ell j}{n}$.  By basic trigonometry we see that $z_\ell = \overline{z_{n-\ell}}$ (complex conjugate) for $1 \leq \ell \leq n-1$.  We clearly get  $z_0 = -1$.  Furthermore, if $n-1$ is odd, then $z_{\frac{n}{2}} = 3$ (if $j$ is odd), while $z_{\frac{n}{2}} = 1$ (if $j$  is even).   The upshot is that the final displayed product equals the integer $-1$, multiplied by  the product of paired complex conjugates, and then (if $n-1$ is odd) multiplied by the positive integer $1$ or $3$.    The displayed conclusion of the Proposition  follows immediately, and yields the final statement as well.     \hfill $\Box$

\begin{remark}\label{K0canbeinfinite}
We note that possibility (2) of Proposition \ref{sizeofK0isH(n)} can indeed occur, for instance,  if $n=6$ and $j=5$. An infinite collection of examples for which  the group $K_0(L_K(C_n^j))$ is infinite is described in \cite{ASch}.
\end{remark}


\section{Leavitt path algebras of the graphs $C_n^0$ and $C_n^1$}\label{SectionC^0andC^1}

In this section we describe the Leavitt path algebras $\{L_K(C_n^0) \ | \ n\in \mathbb{N}\}$ and $\{L_K(C_n^1) \ | \ n\in \mathbb{N}\}$.   Although the analysis of these two collections of algebras is not as intricate as the analysis required to fully describe the algebras $\{L_K(C_n^2) \ | \ n\in \mathbb{N}\}$ (which we carry out in Section \ref{SectionC^2}), the ideas presented in this section will be of interest in their own right, and will help clarify the broader picture.

We begin by reminding the reader about properties of the classical {\it Leavitt algebras} $L_K(1,m)$.    For any integer $m\geq 2$, $L_K(1,m)$  is the free associative $K$-algebra  in $2m$ generators $x_1, x_2, \dots, x_m$, $y_1, y_2, \dots, y_m$, subject to the relations
$$ y_i x_j = \delta_{i,j}1_K \ \ \mbox{and} \ \ \sum_{i=1}^m x_i y_i = 1_K.$$
These algebras were first defined and investigated in \cite{L}, and formed the motivating examples for the more general notion of  Leavitt path algebra.  It is easy to see  that for $m\geq 2$,  if $R_m$ is the graph having one vertex and $m$ loops (the ``rose with $m$ petals" graph), then $L_K(R_m) \cong L_K(1,m)$.    It is clear from the description given in Section \ref{Background}    that each $L_K(R_m)$ is purely infinite simple. It is straightforward from $(**)$ that  $K_0(L_K(R_m)) \cong M_{R_m}^*$ is the cyclic group $\Z_{m-1}$, where the regular module $[L_K(R_m)]$ in $K_0(L_K(R_m))$ corresponds to $ 1$ in  $\Z_{m-1}$.

As we shall see, all of the groups of the form $M_{C_n^0}^*$ or $M_{C_n^1}^*$, 
and some of the groups of the form $M_{C_n^2}^*$,  are cyclic.       Purely infinite simple unital Leavitt path algebras $L_K(E)$ whose corresponding $K_0$ groups are cyclic and for which ${\rm det}(I_{|E^0|}  - A_E^t) \leq 0$ are relatively well-understood, and arise as matrix rings over the Leavitt algebras $L_K(1,m)$, as follows.  Let $d \geq 2$, and consider the graph $R_m^d$ having two vertices $v_1, v_2$; $d-1$ edges from $v_1$ to $v_2$; and $m$ loops at $v_2$:  
$$  R_m^d \ = \   \xymatrix{
  \bullet^{v_1}  \ar[r]^{(d-1)}
 & \bullet^{v_2}  \ar@(ur,dr)[]^{(m)}  } \ \ \ \  \
$$
   It is shown in \cite{AALP} that  the matrix algebra ${\rm M}_d(L_K(1,m))$ is isomorphic to $L_K(R_m^d)$.  By standard Morita equivalence theory we have that 
   $K_0({\rm M}_d(L_K(1,m))) \cong K_0(L_K(1,m)).$
   Moreover, the element $[{\rm M}_d(L_K(1,m))]$ of $K_0({\rm M}_d(L_K(1,m)))$  corresponds to the element $d$ in  $\Z_{m-1}$.   In particular,  the element $[{\rm M}_{m-1}(L_K(1,m))]$ of $K_0({\rm M}_{m-1}(L_K(1,m)))$  corresponds to $m-1 \equiv 0$ in  $\Z_{m-1}$.    Finally, an easy computation yields that ${\rm det} (I_2 - A_{R_m^d}^t) = -(m-1) \leq 0 $ for all $m,d$.    
Therefore, by invoking the Algebraic KP Theorem, the previous discussion immediately yields the following.  

\begin{proposition}\label{isotomatrixoverLeavittprop}
Suppose that $E$ is a graph for which $L_K(E)$ is unital purely infinite simple.  Suppose that $M_E^*$ is isomorphic to the cyclic group  $\Z_{m-1}$, via an isomorphism which takes the element $ \sum_{v\in E^0}[v]$ of $M_E^*$  to the element $d$ of   $\Z_{m-1}$.  Finally, suppose that ${\rm det} (I_{|E^0|} - A_E^t) \leq 0.$  Then
$L_K(E) \cong {\rm M}_d(L_K(1,m)).$ 
\end{proposition}

\begin{corollary}\label{corollarytogetmatrixoverLeavitt}
 Suppose $n\in \mathbb{N}$ and $0\leq j \leq n-1$ are integers for which  the group $K_0(L_K(C_n^j))$ is isomorphic to the cyclic group   $\Z_{m-1}$ for some positive integer $m$.  Then   $L_K(C_n^j) \cong {\rm M}_{m-1}(L_K(1,m)).$
\end{corollary}
{\bf Proof}.   By Proposition \ref{detisnegative} we have that ${\rm det}(I_n - A_{C_n^j}^t) \leq 0$, and by Proposition \ref{sumverticesisidentity} we have that the element $ \sum_{v\in (C_n^j)^0}[v]$  is the zero element of $M_{C_n^j}^*$; that is, $ \sum_{v\in (C_n^j)^0}[v]$  corresponds to the element $m-1$ of  $\Z_{m-1}$.   The result now follows from Proposition \ref{isotomatrixoverLeavittprop}.   \hfill $\Box$

\subsection{Leavitt path algebras of the graphs $C_n^0$}

Let $n\in \mathbb{N}$.   The generating relations for  $M_{C_n^0}^*$ are given by
$$[v_i] = [v_{i+1}] + [v_i]$$
for $1\leq i \leq n$, where subscripts are interpreted ${\rm mod}\ n$.  But as $M_{C_n^0}^*$ is a group, we can cancel in each of these relations to get 
$$[v_{i+1}] = 0 \ \ \mbox{in} \ M_{C_n^0}^*$$
for all $i$.   Since each of the generating elements of $M_{C_n^0}^*$ is therefore $0$, we conclude
that for each $n\geq 1$, the group $M_{C_n^0}^*$ is the one element group $\{0\}$. 
Now applying 
Corollary \ref{corollarytogetmatrixoverLeavitt}, 
we get
\begin{proposition}\label{LeavittofCn0}
For each $n\geq 1$,
$$L_K(C_n^0)\cong L_K(1,2).$$
\end{proposition}

\subsection{Leavitt path algebras of the graphs $C_n^1$}

While we have just seen that  the algebras $\{ L_K(C_n^0) \ | \ n\in \mathbb{N}\}$ are the same up to isomorphism, we show now that the algebras $L_K(C_n^1)$ are pairwise non-isomorphic, and we describe each of these as a matrix ring over a Leavitt algebra.

 Let $n\in \mathbb{N}$.   The generating relations for  $M_{C_n^1}^*$ are given by
$$[v_i] = [v_{i+1}] + [v_{i+1}] = 2[v_{i+1}]$$
for $1\leq i \leq n$, where subscripts are interpreted ${\rm mod}\ n$.  So for each $1 \leq i \leq n$ we have  that 
$$[v_i]  = 2[v_{i+1}] = 4[v_{i+2}] = \cdots  = 2^{n-i}[v_n]= 2^{n+1-i}[v_1].$$
In particular,  each $[v_i]$ is in the subgroup of  $M_{C_n^1}^*$ generated by $[v_1]$.  Since the set $\{[v_i] \ | \ 1\leq i \leq n\}$ generates $M_{C_n^1}^*$, we conclude that  $M_{C_n^1}^*$ is cyclic, and $[v_1]$  (indeed, any $[v_i]$) is a generator.  But by Proposition \ref{sizeofK0isH(n)} and \cite{H} we have  $|K_0(L_K(C_n^1))| = |M_{C_n^1}^*| = H_1(n) = 2^n-1$, so that $M_{C_n^1}^* \cong \Z_{2^n-1}$.   So we are in position to apply Corollary \ref{corollarytogetmatrixoverLeavitt}, which yields

\begin{proposition}
For every $n\in \mathbb{N}$,
$$L_K(C_n^1) \cong {\rm M}_{2^n-1}(L_K(1,2^n)).$$
\end{proposition}

It is well-known that the group $K_0(A)$ is an isomorphism invariant of the ring $A$.   Using this observation, we see that the algebras $\{L_K(C_n^1) \ | \ n\in \mathbb{N}\}$ are pairwise non-isomorphic, as the $K_0$ groups corresponding to different values of $n$ are of different size.


\bigskip

\section{Leavitt path algebras of the graphs $C_n^2$}\label{SectionC^2}

In this final section we analyze the Leavitt path algebras of the form $L_K(C_n^2)$.     Let $n\in \mathbb{N}$.   The generating relations for  $M_{C_n^2}^*$ are given by
$$[v_i] = [v_{i+1}] + [v_{i+2}] $$
for $1\leq i \leq n$, where subscripts are interpreted ${\rm mod }\ n$.
For notational ease we will typically  focus on the generator $[v_1]$ of $M_{C_n^2}^*$;  corresponding to any statement  established for  $[v_1]$ in  $M_{C_n^2}^*$,  there will be (by the symmetry of the relations)  an analogous statement in  $M_{C_n^2}^*$ for each $[v_i]$, $1\leq i \leq n$.

By Proposition \ref{sizeofK0isH(n)} the sequence $H_2(n)$ will clearly play a role in this analysis.   In \cite{H}, Haselgrove establishes that the values of $H_2(n)$ may be defined by setting $H_2(1) = 1, H_2(2) = 1$, and defining the remaining values recursively as:
$$ H_2(n) = H_2(n-1) + H_2(n-2) + 1 - (-1)^n \ \ \mbox{for}  \ n\geq 2.$$
Recast, $H_2(n) = H_2(n-1) + H_2(n-2)$ if $n$ is even, while $H_2(n) = H_2(n-1) + H_2(n-2) + 2$ if $n$ is odd.   The first few terms of  the second Haselgrove sequence $H_2$  are thus
$$1,1,4,5,11,16,29,45,76,121, 199,320, ...$$

This integers which appear in the second Haselgrove sequence are  also known as the {\it associated Mersenne numbers}; this terminology is used both in the Online Encyclopedia of Integer Sequences \cite[Sequence A001350]{OEIS}, and by Haselgrove  in \cite{H}.   The second Haselgrove sequence has appeared in a number of combinatorial contexts, see e.g. \cite{B}.

A central role will be played by the elements of the standard {\it Fibonacci sequence} $F$, defined by setting  $F(1) = 1$, $F(2) = 1$, and $F(n) = F(n-1) + F(n-2)$ for all $n \geq 3$.  (We may also define $F(0) = 0$ consistently with the given recursion equation.)  Of course, $F$ is the well-known sequence 
$$1,1,2,3,5,8,13,21,34,55,89,144, ...$$
We begin by noticing that
$$[v_1] = [v_2] + [v_3] = ([v_3] + [v_4]) + [v_3] = 2[v_3] + [v_4] = 2([v_4] + [v_5]) + [v_4] = 3[v_4] + 2[v_5] = \cdots$$
which inductively gives, for $1\leq i \leq n$, 
$$[v_1] = F(i)[v_i] + F(i-1)[v_{i+1}] $$
in $M_{C_n^2}^*$.    Setting $i=n$, and using that $[v_{n+1}] = [v_1]$ by notational convention, we get in particular that 
$[v_1] = F(n)[v_n] + F(n-1)[v_1]$, so that 
$$  0 = F(n)[v_n] + (F(n-1)-1)[v_1]  \ \mbox{in} \ M_{C_n^2}^*.$$

As is standard, for integers $a,b$, ${\rm gcd}(a,b)$ denotes the greatest common divisor of $a$ and $b$.   A key role will be played by the following integer.

\begin{definition}\label{d(n)definition}
For any positive integer $n$ we define
$$d(n) = {\rm gcd}(F(n), F(n-1)-1).$$
(We often denote $d(n)$ simply by $d$ when appropriate.)
\end{definition}

In particular, the previous displayed equation yields

\begin{lemma}\label{d(n)xis0}
We let $x$ denote the element  
$ \frac{F(n)}{d}[v_n] + \frac{F(n-1)-1}{d}[v_1]$  of  $M_{C_n^2}^*$.   Then $dx = 0$ in $M_{C_n^2}^*$.
\end{lemma}


Our goal will be to achieve an explicit description  of $K_0(L_K(C_n^2))$.  Concretely, if we interpret $\Z_1$ as the trivial group $\{0\}$, then in our main result  (Theorem \ref{TheTheorem})  we will show that $K_0(L_K(C_n^2))\cong \Z_{d(n)}\times \Z_{\frac{H_2(n)}{d(n)}}$ for every $n\in \N$. The procedure to obtain such an isomorphism will be to find the appropriate elements $x$ and $y$ which arise in the next lemma (concretely, $x$ is the element given in Lemma \ref{d(n)xis0}).

\begin{lemma}\label{grouplemma}
Let $G$ be an abelian group with order $O(G)=ab$ for some $a,b\in \N$. Suppose that there exist $x,y\in G$ such that the set $\{x,y\}$ generates $G$, and such that their orders have the property that $O(x)|a$ and $O(y)|b$. Then $G\cong \Z_a\times \Z_b$.
\end{lemma}
\begin{proof}
Consider the set $S=\{mx+ny\ | \  m,n\in  \Z\}$. By hypothesis $O(x)|a$, so we get that $|\{mx \ | \  m\in \Z\}|\leq a$, and similarly $|\{ny \ | \  n\in \Z\}|\leq b$. Then $|S|\leq ab$. But the fact that the set $\{x,y\}$ generates $G$ gives $G\subseteq S$ so that $ab=O(G)\leq |S|\leq ab$. Now if either $O(x)<a$ or $O(y)<b$ then the last inequalitly would be strict, that is, $|S|< ab$, producing a contradiction. This gives that $O(x)=a$ and $O(y)=b$. Now the map which sends $(1,0)\in \Z_a\times \Z_b$ to $x\in G$ and $(0,1)\in \Z_a\times \Z_b$ to $y\in G$ can be easily checked to be a group isomorphism.
\end{proof}

The elements $y$ we are looking for in order to apply Lemma \ref{grouplemma} will depend on the values of $n$, and before we can actually find them, we will need to recall and establish several formulas which will be heavily used in the sequel,  at times without further mention.

\begin{proposition}\label{FibonacciFormulas}(see e.g. \cite{K})  For all $n,m\in \N$ we have:
\begin{equation*}
\begin{split}
(-1)^n & = F(n+1)F(n-1)-F(n)^2   \hskip1.8cm     \text{\it (Cassini's Identity)} \\
F(2n-1) & =  F(n)^2+F(n-1)^2  \hskip3cm  \text{\it (F-odd)}\\ 
F(2n) & = (F(n-1)+F(n+1))F(n)  \hskip1.7cm  \text{\it (F-even)}\\
\gcd(F(n),F(m)) & = F(\gcd(n,m))  \hskip3.9cm  \text{\it (F-gcd)} \\
F(n) \text{ is even } & \Leftrightarrow n \equiv 0 \ {\rm mod} \ 3  \hskip3.8cm \text{\it (even F-values)}  \\
F(n + 2)^2 & = 3 F(n + 1)^2 - F(n)^2 - 2 (-1)^n   \hskip1.3cm \text{(Hoggatt71)} \\
F(n + 1)^2 - F(n)^2  & = F(n + 2) F(n - 1)  \hskip3.1cm \text{(Vajda12)} \\
F(n) F(n + 1) &  = F(n - 1) F(n + 2) + (-1)^{n-1}   \hskip1.4cm \text{(Vajda20)}\\
H_2(n) & = F(n+1) + F(n-1) -1 - (-1)^n  \hskip.7cm \text{(HtoF)} 
\end{split}
\end{equation*}
\end{proposition}

\smallskip

\begin{remark} 
Formula (HtoF) of Proposition \ref{FibonacciFormulas} appears in \cite[A001350]{OEIS} without proof.  However, this formula is easy to establish, using induction together with the defining equations of the sequences $H_2(n)$ and $F(n)$.  
\end{remark} 
It will be helpful to have  a more computationally explicit description of the number $d(n)$ than that given in Definition \ref{d(n)definition}.   To get it (and to establish a number of other results in the sequel as well), we  will often use the following well-known property of  greatest common divisors: 
\begin{equation}
\gcd(a,b)=\gcd(a,b+ka)=\g(a+mb,b) \ \text{ for every }a,b,k,m\in \Z. \tag{ $\overset{*}{=}$}
 \end{equation}
 For notational ease, each time we use that fact in our proofs we will note it by $\overset{*}{=}$.

\begin{lemma}\label{gcdreductionformula}
Let $n\in \N$, and let $j\in \N$ with $j+2 \leq n$.   Then \small
$$d(n) =
 \  {\rm gcd}( \ F(n-(j+1)) + (-1)^{j+1}F(j+1) \ , \  F(n-(j+2)) + (-1)^jF(j+2) \ ).$$
\end{lemma}
\begin{proof}
We fix $n$, and use finite induction on $j$.   We first note that 
\begin{align*}
d(n) = {\rm gcd}(F(n), F(n-1)-1)  & =  {\rm gcd}(F(n-1)+F(n-2), F(n-1)-1)  \\
&  \overset{*}{=}   {\rm gcd}(F(n-2) + 1,  F(n-1)-1),
\end{align*}
with the starred equality coming from fact $(\overset{*}{=})$ by writing  
$$F(n-1) + F(n-2) = (F(n-2) + 1) +  1 \cdot [F(n-1) - 1]  $$ and defining $a = F(n-2) + 1, m=1,$ and $b=F(n-1)-1$.  
But the expression $${\rm gcd}(F(n-1)-1, F(n-2) + 1)$$ is precisely the $j=0$ case of the proposed formula, since  $F(1) = F(2) = 1$. 

Suppose now that the formula is true for some integer $j$ having $j+3 \leq n$; we establish that the formula is true for the integer $j+1$.  But using the same idea as in the $j=0$ case, we first write $F(n-(j+1)) = F(n-(j+2)) + F(n-(j+3))$, then write  
$$\hskip-7cm F(n-(j+2)) + F(n-(j+3)) + (-1)^{j+1}F(j+1) $$
$$ =  (F(n-(j+3)) - (-1)^{j}F(j+2) + (-1)^{j+1}F(j+1)) + 1\cdot [F(n-(j+2)) + (-1)^jF(j+2)],$$ \normalsize
and then use fact $(\overset{*}{=})$ to get that 
$$\hskip-3cm {\rm gcd}( \ F(n-(j+1)) + (-1)^{j+1}F(j+1) \ , \  F(n-(j+2)) + (-1)^jF(j+2) \ )$$
\begin{align*} 
&  \overset{*}{=} {\rm gcd}( \  F(n-(j+3)) + (-1)^{j+1}F(j+2) + (-1)^{j+1}F(j+1) \ , \  F(n-(j+2)) + (-1)^{j}F(j+2) \ ) \\ 
 & = {\rm gcd}( \ F(n-(j+2)) + (-1)^{j}F(j+2) \ , \  F(n-(j+3)) + (-1)^{j+1}F(j+3) \ ),  \ \ \ \normalsize
 \end{align*}
which is the appropriate formula for $j+1.$   

We note for later that the equation $F(n-(j+1)) = F(n-(j+2)) + F(n-(j+3))$ is valid in the case $j+3 = n$, since by definition we have $F(0)=0$.  
\end{proof}

\begin{proposition}\label{d(n)formulas} For any $n\in \N$ let $d(n)$ denote ${\rm gcd}(F(n), F(n-1)-1)$.  Let $m\in \N\cup\{0\}$.  Then  \begin{align*}
d(2m+1)  & =\begin{cases}
1 & \text{if }\ 2m+1\equiv 1 \text{ or }5\ \text{\rm mod}\ 6\\
2 & \text{if }\ 2m+1\equiv 3\ \text{\rm mod}\ 6\end{cases}\\
d(2m+2)  & =\begin{cases}
F(m)+F(m+2) & \text{if }\ m \text{ is even}\\
F(m+1) & \text{if }\ m \text{ is odd}
\end{cases}
\end{align*}
These formulas thereby provide a description of the integer  $d(n)$ for all integers $n\geq 1$.
\end{proposition}
\begin{proof}
We first suppose that $n$ is odd, and write $n = 2m+1$ for $m\in \N$.   
By definition, $d(2m+1) = {\rm gcd}(F(2m+1), F(2m) -1)$.   Using Lemma \ref{gcdreductionformula} in the case $j = 2m-1$, we get 
\begin{align*}
&  d(2m+1)  \\
\ &  = {\rm gcd}( \ F((2m+1)-(2m)) + (-1)^{2m}F(2m) \ , \  F((2m+1)-(2m+1)) + (-1)^{2m-1}F(2m+1) \ ) \\
&  = {\rm gcd}(F(1) + F(2m), F(0) + (-1)F(2m+1)) \\
& =  {\rm gcd}(1 + F(2m), - F(2m+1)) \\ 
& =  {\rm gcd}(1 + F(2m), F(2m+1)),
\end{align*}
with the final equality coming from the basic number theory fact that ${\rm gcd}(a,-b) = {\rm gcd}(a,b)$ for any integers $a,b$.  So we have
$$ {\rm gcd}(F(2m+1), F(2m) -1) = d(2m+1) =  {\rm gcd}(1 + F(2m), F(2m+1)) .$$
But if we let $g$ denote $F(2m+1)$ and $h$ denote $F(2m)-1$, we have thus established that ${\rm gcd}(g,h) = {\rm gcd}(h+2,g)$.   But again using a basic number theory fact, this last configuration can only occur when ${\rm gcd}(g,h) = 1$ or ${\rm gcd}(g,h) = 2$.   Clearly the former occurs when at least one of $g,h$ is odd, while the latter occurs when both $g$ and $h$ are even.   

Summarizing, we have established that in case $n=2m+1$, then $d(n) = 1$ when at least one of $F(n)$ and $F(n-1) - 1$ is odd, and $d(n) = 2$ when both $F(n)$ and $F(n-1)-1$ are even.    By Proposition \ref{FibonacciFormulas},  both $F(n)$ and $F(n-1)-1$ are even exactly when $n \equiv 0 \ {\rm mod} \ 3$.   For odd $n$, $n \equiv 0 \ {\rm mod} \ 3$ precisely when $n \equiv 3 \ {\rm mod} \ 6$.   This establishes the formulas for $d(n)$ when $n$ is odd.

\smallskip

Now suppose $n$ is even, and write $n=2m+2$ for some $m\in \N$.  We again use Lemma \ref{gcdreductionformula}; in this situation we let $j=m$, which yields  
\begin{align*}
 & d(2m+2) \\
 &  =  {\rm gcd}( \ F((2m+2)-(m+1)) + (-1)^{m+1}F(m+1) \ , \  F((2m+2)-(m+2)) + (-1)^{m}F(m+2) \ ) \\
 & = {\rm gcd}( \ F(m+1) + (-1)^{m+1}F(m+1) \ ,  \  F(m) + (-1)^{m}F(m+2)\ ) \\
 &  =\begin{cases}
 {\rm gcd}( \ 0 \ ,  \  F(m) + F(m+2)\ )& \text{if }\ m \text{ is even}\\
{\rm gcd}( \ 2F(m+1)  \ ,  \  F(m) - F(m+2)\ ) & \text{if }\ m \text{ is odd}.
\end{cases} 
\end{align*}
\noindent
But ${\rm gcd}(0,a) = a$ for any positive integer $a$.   Moreover, since $F(m) - F(m+2) =  -F(m+1),$  ${\rm gcd}(a,-b) = {\rm gcd}(a,b)$, and ${\rm gcd}(2a,a) = a$, we get
\smallskip
 
$ \hskip-.5cm  =\begin{cases}
 F(m) + F(m+2)\ & \text{if }\ m \text{ is even}\\
F(m+1) & \text{if }\ m \text{ is odd}.
\end{cases}$

\smallskip
\noindent
This completes the proof of the Proposition. \end{proof}

In the following result we establish  key relationships between  $d(n)^2$ and $H_2(n)$.

\begin{proposition}\label{d^2|H_2} Let $n\in \N$. Then $d(n)^2|H_2(n)$. Furthermore, if $n\equiv 0 \m 4$ then $H_2(n)=5d(n)^2$, whereas if $n\equiv 2 \ \text{\rm mod } \ 4$ then $H_2(n)=d(n)^2$.
\end{proposition}

\begin{proof}
We use the formulas provided in Proposition \ref{d(n)formulas}.   If $n\equiv 1$ or $n \equiv 5$ ${\rm mod} \ 6$ then $d(n)=1$ and the result is obvious.     Suppose $n\equiv 3 \m  6$, so that $d(n)=2$.  Then 
\begin{align*} H_2(n) & =F(n+1)+F(n-1)-1-(-1)^n  \ \  \ \ \ \ {\rm by \ (HtoF)} \\
& =F(n+1)+F(n-1) \ \ \hskip2.5cm  {\rm as } \ n \ {\rm is \ odd} \\
 & =F(n)+2F(n-1)  \\
 & =3F(n-1)+F(n-2) \\
 & =4F(n-2)+3F(n-3).
\end{align*} 
Thus  $d(n)^2=4$ is a divisor of $H_2(n)$ if and only if $4$ is a divisor of $F(n-3)$. Using formula (F-gcd)  we get $$\gcd(8,F(n-3))=\gcd(F(6),F(n-3))=F(\gcd(6,n-3))=F(6)=8.$$ In other words, $8|F(n-3)$ and the claim follows.   Thus we have established the assertion for $n$ odd.

When $n$ is even there are two possibilities.    First, suppose $n=2m+2$ with $m$ odd.   Then $d(n) = F(m+1)$ by Proposition \ref{d(n)formulas}.   On the other hand, 
\begin{align*}
H_2(n)  & = H_2(2m+2)  \\
& = F(2m+3) + F(2m+1) -1 - (-1)^{2m+2}  \hskip1cm {\rm by \ (HtoF)} \\
& = F(2m+3) + F(2m+1)-2.
\end{align*}
    We seek to show that $H_2(n) = 5d(n)^2$; i.e., we seek to show that
$$F(2m+3) + F(2m+1)-2 = 5F(m+1)^2. $$
Using formula (F-odd) twice on the left hand side, this is equivalent to showing that
$$F(m+2)^2 + 2F(m+1)^2  + F(m)^2-2 = 5F(m+1)^2,$$
which holds if and only if
$$F(m+2)^2  + F(m)^2-2 = 3F(m+1)^2.$$
But since $m$ is odd, this last assertion is indeed true by formula (Hoggatt71).  

Now suppose $n=2m+2$ with $m$ even.   Then $d(n) = F(m) + F(m+2)$ by Proposition \ref{d(n)formulas}, while $H_2(n) = H_2(2m+2) = F(2m+3) + F(2m+1) -1 - (-1)^{2m+2} = F(2m+3) + F(2m+1)-2.$  In this case we seek to show that $H_2(n) = d(n)^2$; i.e., we seek to show that 
$$F(2m+3) + F(2m+1)-2 = (F(m) + F(m+2))^2.$$
Expanding  the right side of the displayed equation and using  $F(m+2) = F(m) + F(m+1)$, we easily get
$$(F(m) + F(m+2))^2 = 4F(m)^2 + 4F(m)F(m+1) + F(m+1)^2.$$
On the other hand, as in the previous paragraph, we get 
$$F(2m+3) + F(2m+1)-2 = F(m+2)^2 + 2F(m+1)^2  + F(m)^2-2.$$
Thus we seek to show that 
$$F(m+2)^2 + 2F(m+1)^2  + F(m)^2-2 = 4F(m)^2 + 4F(m)F(m+1) + F(m+1)^2,$$
which is clearly equivalent to showing that
$$F(m+2)^2 + F(m+1)^2 - 3 F(m)^2 - 4F(m)F(m+1) - 2 = 0.$$
But
\begin{align*}
& \hskip-1cm F(m+2)^2 + F(m+1)^2 - 3 F(m)^2 - 4F(m)F(m+1) - 2  \\
 & = (F(m) + F(m+1))^2 + F(m+1)^2 - 3 F(m)^2 - 4F(m)F(m+1) - 2 \\
 & = 2F(m+1)^2 - 2 F(m)^2 - 2F(m)F(m+1) - 2 \\
 & = 2 F(m+2)F(m-1) -2F(m)F(m+1) - 2 \hskip1cm \text{(Vajda12)} \\
 & = 2 \cdot [- (-1)^{m-1} ] -2   \hskip1cm \text{(Vajda20)} \\
 &  = 2 \cdot 1 -2 \hskip1cm \text{(since $m$ is even)} \\
 & =0
\end{align*}
as desired. 
\end{proof}

Since our goal will be to show that $K_0(L_K(C_n^2))\cong \Z_{d(n)}\times \Z_{\frac{H_2(n)}{d(n)}}$, and Proposition \ref{d^2|H_2} gives that $d(n)|\frac{H_2(n)}{d(n)}$, then clearly in $K_0(L_K(C_n^2))$ every element will have order at most  $\frac{H_2(n)}{d(n)}$. As such, one of the important steps will be to show that, for every $1\leq i\leq n$, we have $$\text{Step 1. }\ \ \frac{H_2(n)}{d(n)}[v_i]=0 \text{ in }M^*_{C_n^2}.$$

The other major step deals with finding the appropriate $y$ as appears in Lemma \ref{grouplemma}. Concretely if $x$ is the element $\frac{F(n)}{d(n)}[v_n] + \frac{F(n-1)-1}{d(n)}[v_1]$, then $$\text{Step 2. }\ \ \text{Either }\{x,[v_1]\} \text{ or } \{x,[v_n]\} \text{ generate }M^*_{C_n^2}.$$

Before we tackle Step 1, we will elaborate Step 2 a bit further, laying down some of the equations that are sufficient to guarantee that the aformentioned sets generate $M^*_{C_n^2}$. In doing so we will end up obtaining some more properties of greatest common divisors of Fibonacci and Halselgrove numbers, one of which will in fact be needed to complete our proof of Step 1.

Concretely, suppose we would like to prove that the set $\{x,[v_1]\}$ generates $M^*_{C_n^2}$. This will be equivalent to showing that we can express $[v_n]$ as a combination of $x$ and $[v_1]$ because in that case we would be also able to generate $[v_{n-1}]=[v_n]+[v_1]$, and then all the $[v_i]$ (which are themselves a set of generators by definition) just by going backwards recursively.

So we wish to find integers $p,q$ for which $[v_n] = px + q[v_1]$ in  $M^*_{C_n^2}$, i.e., integers $p,q$ for which 
$$ p\frac{F(n)}{d(n)}[v_n]+p\frac{F(n-1)-1}{d(n)}[v_1]+q[v_1]=[v_n]. $$
Now taking into account Proposition \ref{sizeofK0isH(n)}, the previous equation will have a solution in $(p,q)$ if (perhaps not only if) the following system of congruences has a solution:
\begin{equation} \tag{\dag}
\begin{cases}
p\frac{F(n)}{d(n)} & \equiv 1\m  H_2(n)   \\
p\frac{F(n-1)-1}{d(n)}+q & \equiv 0 \m  H_2(n) .
\end{cases}
\end{equation}
This system in turn will have a solution if and only if $\gcd(\frac{F(n)}{d(n)},H_2(n))=1$. Unfortunately this does not always happen, as we see in the following result:

\begin{proposition}\label{FirstFormulaForStep2} Let $n\in \N$. Then 
$$\gcd\left(\frac{F(n)}{d(n)},H_2(n)\right)=\begin{cases}
1 & \text{ if }n \not\equiv 0 \m   6   \\
2 & \text{ if }n \equiv 0 \m  6 
\end{cases}$$
\end{proposition}
\begin{proof}
The proof necessarily distinguishes several cases for $n$ (essentially those given by the formula of $d(n)$ in Proposition \ref{d(n)formulas}), and systematically uses the various Fibonacci formulas detailed in Proposition \ref{FibonacciFormulas} as well. In addition, we will write $A(n)$ instead of the longer $\gcd\left(\frac{F(n)}{d(n)},H_2(n)\right)$. 

Our approach consists in finding self-recursive formulas with some periodicity (typically $6$ or $12$) to get down to some computable cases.
\vskip0.2cm

\underline{Case 1:} $n\equiv 1\text{ or }5\m \ 6$.

In this situation Proposition \ref{d(n)formulas} gives that $d(n)=1$. Suppose that $n\geq 7$, then:
\begin{align*}
A(n) = & \gcd(F(n),F(n+1)+F(n-1) - 1 - (-1)^n)=\gcd(F(n),F(n)+2F(n-1))\\
\eg & \gcd(F(n),2F(n-1))=\gcd(F(n-1)+F(n-2),2F(n-1))\\
\eg & \g(F(n-1)+F(n-2),-2F(n-2))=  \g(2F(n-2)+F(n-3),-2F(n-2))\\
\eg & \g(F(n-3),-2F(n-2))) =  \g(F(n-3),2F(n-2))) \\
= & \g(F(n-3),2F(n-3)+2F(n-4)) \eg  \g(F(n-3),2F(n-4))
\end{align*}
At this point it might seem like we have obtained a 3-cycle recursion, but it is not quite the case as $n$ is odd while $n-3$ is even, and this impacts the formula for $H_2(n)$ or $H_2(n-3)$. However, repeating the previous steps 3 more times we do get:
\begin{align*}
A(n) = & \g(F(n-3),2F(n-4))=\dots=\g(F(n-6),2F(n-7)) \\
\eg & \g(F(n-6),F(n-6)+2F(n-7)) =  \g(F(n-6),F(n-5)+F(n-7)+0)\\
= & \g(F(n-6),H_2(n-6)) = A(n-6), \text{ for all }n\geq 7.
\end{align*}
As one quickly checks that $A(1)=A(5)=1$, this case is completed.
\vskip0.2cm

\underline{Case 2:} $n\equiv 3\m 6$.

Now by Proposition \ref{d(n)formulas} we know that $d(n)=2$ and by definition of $d(n)$ we know that $\frac{F(n)}{d(n)}\in \Z$, which allows us to perform the following computations:
\begin{align*}
A(n) = & \g(\tfrac{F(n)}{2},F(n+1)+F(n-1)- 1 - (-1)^n)=\g(\tfrac{F(n)}{2},F(n)+2F(n-1))\\
  \eg  & \g(\tfrac{F(n-1)+F(n-2)}{2},2F(n-1))\eg \g(\tfrac{2F(n-2)+F(n-3)}{2},-2F(n-2))\\
  \eg  & \g(\tfrac{3F(n-3)+2F(n-4)}{2},F(n-3)) \eg \g(\tfrac{F(n-3)+2F(n-4)}{2},F(n-3)) \\
  \eg  & \g(\tfrac{3F(n-4)+F(n-5)}{2},-2F(n-4)) \eg  \g(\tfrac{-F(n-4)+F(n-5)}{2},-2F(n-4))\\
  \eg  & \g(\tfrac{-F(n-6)}{2},-2F(n-5))= \g(\tfrac{F(n-6)}{2},2F(n-6)+2F(n-7))\\
  \eg  & \g(\tfrac{F(n-6)}{2},F(n-6)+2F(n-7)) = \g(\tfrac{F(n-6)}{2},F(n-5)+F(n-7)+0)=A(n-6)  
\end{align*}
for all $n\geq 9$. By checking that $A(3)=1$, this case is also finished.   
\vskip0.2cm

\underline{Case 3:} $n\equiv 0\m 4$.

In this case we have that $n=2m+2$ for some $m$ odd, so that Proposition \ref{d(n)formulas} gives that $d(n)=d(2m+2)=F(m+1)$. Also, using the formulas in Proposition \ref{FibonacciFormulas} we know by (F-even) that $$F(n)=F(2m+2)=F(2(m+1))=(F(m)+F(m+2))F(m+1)$$ and by (F-odd) that
$$F(2m+3)=F(2(m+2)-1)=F(m+2)^2+F(m+1)^2\text{ and }F(2m+1)=F(m+1)^2+F(m)^2,$$ whereas Cassini's identity gives $$(-1)^{m+1}=1=F(m+2)F(m)-F(m+1)^2$$
Putting all these facts together we get
\begin{align*}
A(n) = & \g(F(m)+F(m+2),F(2m+3)+F(2m+1)-1-(-1)^{2m+2})\\
     = & \g(F(m)+F(m+2),F(m+2)^2+2F(m+1)^2+F(m)^2-2)\\
     = & \g(F(m)+F(m+2),F(m+2)^2+2F(m+2)F(m)+F(m)^2-4)\\
     = & \g(F(m)+F(m+2),(F(m)+F(m+2))^2-4)\eg \g(F(m)+F(m+2),4)\\
     = & \g(F(m+1)+2F(m),4)= \g(3F(m)+F(m-1),4)= \g(4F(m-1)+3F(m-2),4)\\
   \eg & \g(3F(m-2),4)=\g(F(m-2),4)= \g(F(m-3)+F(m-4),4)\\
     = & \g(2F(m-4)+F(m-5),4) = \g(3F(m-5)+2F(m-6),4)\\
     = & \g(5F(m-6)+3F(m-7),4) = \g(8F(m-7)+5F(m-8),4) \eg \g(F(m-8),4).
\end{align*}
This gives the 6-cycle recursion $\g(F(m-2),4)=\g(F(m-8),4)$ or, in terms of $n$, a $12$-cycle recursion, so we just need to check enough base values of $A(n)$ for small $n$. They are:
$$A(4)=A(8)=1; \  A(12)=2; \ A(16)=A(20)=1; \ A(24)=2,Ê\text{ etc.}$$
\vskip0.2cm

\underline{Case 4:} $n\equiv 2\m 4$.

In this situation we can write $n=2m+2$ with $m$ even and then Proposition \ref{d(n)formulas} yields $d(n)=d(2m+2)=F(m)+F(m+2)$. Using again the formulas in Proposition \ref{FibonacciFormulas} which we displayed in Case 3 we can obtain:
\begin{align*}
A(n) = & \g(F(m+1),F(2m+3)+F(2m+1)-1-(-1)^{2m+2})\\
     = & \g(F(m+1),F(m+2)^2+2F(m+1)^2+F(m)^2-2) \\
   \eg & \g(F(m+1),(F(m+1)+F(m))^2+F(m)^2-2) \\
   =   & \g(F(m+1),F(m+1)(F(m+1)+2F(m))+2F(m)^2-2) \eg \g(F(m+1),2F(m)^2-2)
\end{align*}     
Now because $m$ is even, Cassini's formula gives
$$(-1)^{m}=1=F(m+1)F(m-1)-F(m)^2$$
so that we can continue with the simplification to get $$A(n) = \g(F(m+1),2F(m+1)F(m-1)-4) \eg \g(F(m+1),4),$$ which would eventually lead to a $6$-cycle in $m$ (or a $12$-cycle in $n$) just as in previous Case 3. Thus, only the smallest values of $A(n)$ need to be checked, and they are the following:
$$A(2)=1; \  A(6)=2; \ A(10)=A(14)=1; \ A(18)=2,Ê\text{ etc.}$$
This completes the proof.\end{proof}

By  Proposition \ref{FirstFormulaForStep2} and the discussion prior to its statement, we immediately get 
\begin{corollary}\label{Step2fornnotequiv0mod6}   \ {\rm (Step 2 in case $n\not\equiv 0$ ${\rm mod}  \ 6$.)}   
Suppose $n\not\equiv 0 $ ${\rm mod} \ 6$.   Then the set $\{x, [v_1]\}$ generates $M^*_{C_n^2}$.
\end{corollary}

Later in the paper we will come back to the consideration of the system of congruences needed to obtain Step 2 in the case that $n\equiv 0\m 6$ (that is, when the system (\dag) does not have a solution). But before we get there, we will actually need to have proved Step 1, which we now are able to do, by making use of Proposition \ref{FirstFormulaForStep2}.

\begin{proposition}\label{Step1}  \ {\rm (Step 1.)}  \ Let $n\in \N$. Any element in $M^*_{C_n^2}$ has order dividing $\frac{H_2(n)}{d(n)}$.
\end{proposition}
\begin{proof}
Since the vertices $\{[v_i]\}$ generate $M^*_{C_n^2}$, it will suffice to check that $\frac{H_2(n)}{d(n)}[v_i]=0$ in $M^*_{C_n^2}$, and since this depends on $d(n)$, we will again need to have different cases depending on the values of $n$ in order to use the formula in Proposition \ref{d(n)formulas}.
\vskip0.2cm

\underline{Case 1:} $n\equiv 1\text{ or }5\m 6$.

This case is trivial as Proposition \ref{d(n)formulas} yields $d(n)=1$ so then an application of Proposition \ref{sizeofK0isH(n)} does it.
\vskip0.2cm

\underline{Case 2:} $n\equiv 3\m 6$.

Now Proposition \ref{d(n)formulas} gives $d(n)=2$. The first step will be to show that $d(n)^2=4$ divides $F(n-1)-1$. Write $n = 6m+3$ and proceed by induction on $m$.  The base case gives:  $$F(3-1) - 1 = F(2)-1 = 1-1=0.$$ So now suppose $4\ |\ F((6m+3) - 1)-1$. We will show that $4\ |\ F((6(m+1)+3)-1)-1$.   We have 
\begin{align*}
F(6m+8) - 1 = & F(6m+6) + F(6m+7) -1 = F(6m+5) + 2F(6m+6) - 1 =\dots\\
            = & 5F(6m+2) + 8F(6m+3) - 1 = 8F(6m+3) + 4F(6m+2) + (F(6m+2)-1),
\end{align*}            
which is divisible by $4$ using the induction hypothesis. 

As we have seen, in $K_0(L_K(C_n^2))$ we have $F(n)[v_n] =-(F(n-1)-1)[v_1]$. Multiply both sides by $\frac{H_2(n)}{d(n)^2}$ (which is an integer by Proposition \ref{d^2|H_2}) to get
$$ \frac{H_2(n)}{d(n)^2} F(n)[v_n]= -\frac{H_2(n)}{d(n)^2} (F(n-1)-1)[v_1].$$
We have also just shown that $ \frac{F(n-1)-1}{d(n)^2}$ is an integer, and of course $\frac{F(n)}{d(n)}$ is an integer, so we can rewrite this as
$$ \frac{H_2(n)}{d(n)} \frac{F(n)}{d(n)} [v_n] =-H_2(n) \cdot \frac{F(n-1)-1}{d(n)^2}[v_1].$$

But by Proposition \ref{sizeofK0isH(n)} we know $|K_0(L_K(C_n^2))| = H_2(n)$, so the right hand side is $0$ in $K_0(L_K(C_n^2))$.  So we have
$$ \frac{H_2(n)}{d(n)} \frac{F(n)}{d(n)} [v_n] = 0$$
in $K_0(L_K(C_n^2))$.    But then this means that the order of $[v_n]$ divides $\frac{H_2(n)}{d(n)} \frac{F(n)}{d(n)}$.    Now use Proposition \ref{FirstFormulaForStep2} to conclude that the order of $[v_n]$ divides $\frac{H_2(n)}{d(n)}$.
This finishes this case.
\vskip0.2cm

\underline{Case 3:} $n\equiv 2\m 4$.

An application of Proposition \ref{d^2|H_2} gives that $H_2(n)=d(n)^2$, so we will need to check that $d(n)[v_i]=0$ for all $1\leq i\leq n$. Use Proposition \ref{d(n)formulas} to get that $d(n)=F(\frac{n}{2}+1)+F(\frac{n}{2}-1)$. We make use of the Fibonacci formulas in Proposition \ref{FibonacciFormulas} to get: $$F(n)=F(2\tfrac{n}{2})=(F(\tfrac{n}{2}-1)+F(\tfrac{n}{2}+1))F(\tfrac{n}{2})=d(n)F(\tfrac{n}{2}),$$ and also, taking into account that $\tfrac{n}{2}$ is odd, we obtain
\begin{align*}
F(n-1)-1 = & F(2\tfrac{n}{2}-1)-1=F(\tfrac{n}{2})^2+F(\tfrac{n}{2}-1)^2-1\\
         = & (F(\tfrac{n}{2}+1)F(\tfrac{n}{2}-1)-(-1)^{\tfrac{n}{2}})+F(\tfrac{n}{2}-1)^2-1\\
         = & (F(\tfrac{n}{2}+1)+F(\tfrac{n}{2}-1))F(\tfrac{n}{2}-1)=d(n)F(\tfrac{n}{2}-1)
\end{align*}
As we have seen before, in $M^*_{C_n^2}$ we have the equation $F(n)[v_i]=-(F(n-1)-1)[v_{i+1}]$ where the indices are interpreted mod $n$, which in view of the previous equations, yields:
$$d(n)F(\tfrac{n}{2})[v_i]=-d(n)F(\tfrac{n}{2}-1)[v_{i+1}]\text{ for all }1\leq i\leq n,  \text{ where we understand }[v_{n+1}]\text{ as }[v_1].$$
We will use those equations as well as $[v_i]=[v_{i+1}]+[v_{i+2}]$ in what follows. 
\begin{align*}
-d(n)F(\tfrac{n}{2}-1)[v_1] = & -d(n)F(\tfrac{n}{2}-1)[v_2]-d(n)F(\tfrac{n}{2}-1)[v_3]=-d(n)F(\tfrac{n}{2}-1)[v_2]+d(n)F(\tfrac{n}{2})[v_2]\\
   = & d(n)F(\tfrac{n}{2}-2)[v_2].
\end{align*}
By symmetry of the subindices we can also conclude that $-d(n)F(\tfrac{n}{2}-1)[v_3]=d(n)F(\tfrac{n}{2}-2)[v_4]$ which allows us to get
\begin{align*}
-d(n)F(\tfrac{n}{2}-2)[v_2] = & -d(n)F(\tfrac{n}{2}-2)[v_3]-d(n)F(\tfrac{n}{2}-2)[v_4]\\
= & -d(n)F(\tfrac{n}{2}-2)[v_3]+d(n)F(\tfrac{n}{2}-1)[v_3] = d(n)F(\tfrac{n}{2}-3)[v_3]
\end{align*}
In this fashion it can be shown by induction  that $$d(n)F(\tfrac{n}{2}-i)[v_j]=-d(n)F(\tfrac{n}{2}-(i+1))[v_{j+1}]\text{ for all }1\leq i\leq \tfrac{n}{2}-1\text{ and all }1\leq j\leq n$$
In particular when we are down to step $i=\tfrac{n}{2}-2$ we get $d(n)F(2)[v_j]=-d(n)F(1)[v_{j+1}]$, that is, $d(n)[v_j]=-d(n)[v_{j+1}]$, which in turn gives $$d(n)[v_{j-1}]=d(n)[v_{j}]+d(n)[v_{j+1}]=0,$$ as was needed.
\vskip0.2cm

\underline{Case 4:} $n\equiv 0\m 4$.

In this case Proposition \ref{d^2|H_2} gives that $H_2(n)=5d(n)^2$, so our goal this time will be to show that $5d(n)[v_i]=0$ for all $1\leq i\leq n$. Again apply Proposition \ref{d(n)formulas} to obtain the specific value for $d(n)$, concretely, $d(n)=F(\frac{n}{2})$ and, with similar arguments to those in the previous case we get:
$$F(n)=d(n)(F(\tfrac{n}{2}+1)+F(\tfrac{n}{2}-1))\ \text{ and }\ F(n-1)-1=d(n)(F(\tfrac{n}{2})+F(\tfrac{n}{2}-2))$$
So by symmetry of the indices we can say that 
$$d(n)(F(\tfrac{n}{2}+1)+F(\tfrac{n}{2}-1))[v_j]=-d(n)(F(\tfrac{n}{2})+F(\tfrac{n}{2}-2))[v_{j+1}]\text{ for all }1\leq j\leq n \m n.$$
We seek to get a descending process again, but this time two equations with small numbers in the variables will be needed instead of just one as in the previous case.
\begin{align*}
-d(n)(F(\tfrac{n}{2})+F(\tfrac{n}{2}-2))[v_1]=  & -d(n)(F(\tfrac{n}{2})+F(\tfrac{n}{2}-2))[v_2]-d(n)(F(\tfrac{n}{2})+F(\tfrac{n}{2}-2))[v_3]\\
      = & -d(n)(F(\tfrac{n}{2})+F(\tfrac{n}{2}-2))[v_2]+d(n)(F(\tfrac{n}{2}+1)+F(\tfrac{n}{2}-1))[v_2]\\
      = & d(n)(F(\tfrac{n}{2}-1)+F(\tfrac{n}{2}-3))[v_2]
\end{align*}
By symmetry, the roles of $[v_1]$ and $[v_2]$ can be played by any two consecutive generators, which allows us to proceed inductively until we get, for all $1\leq i\leq \tfrac{n}{2}-3$ and all $1\leq j\leq n\m n$ that
$$ d(n)(F(\tfrac{n}{2}-i)+F(\tfrac{n}{2}-(i+2)))[v_j]=-d(n)(F(\tfrac{n}{2}-(i+1))+F(\tfrac{n}{2}-(i+3)))[v_{j+1}]$$
Substituting in the previous equations the value $i=\tfrac{n}{2}-4$ we get
$$d(n)(F(4)+F(2))[v_j]=-d(n)(F(3)+F(1))[v_{j+1}], \ \mbox{so that} \   4d(n)[v_j]=-3d(n)[v_{j+1}].$$
Similarly,  by substituting instead the value  $i=\tfrac{n}{2}-3$ we get  
$$d(n)(F(3)+F(1))[v_j]=-d(n)(F(2)+F(0))[v_{j+1}]\,  \ \mbox{so that} \ \ 3d(n)[v_j]=-d(n)[v_{j+1}].$$
Using the last two equations we finally obtain
$$-3d(n)[v_j]=4d(n)[v_{j-1}]=4d(n)[v_j]+4d(n)[v_{j+1}]=4d(n)[v_j]-4\cdot 3d(n)[v_j]=-8d(n)[v_j].$$
But this yields $5d(n)[v_j]=0$, as we wanted to prove.
\end{proof}

This completes the proof of the aforementioned Step 1. Now we turn our attention again to Step 2.   We note that we have already established Step 2  in Corollary \ref{Step2fornnotequiv0mod6}   for the cases $n \not\equiv 0\m 6$. Concretely, we showed that $\{x,[v_1]\}$ generates $M^*_{C_n^2}$.

Therefore, only the $n\equiv 0\m 6$ case remains, and for it we will show that $\{x,[v_n]\}$ generates $M^*_{C_n^2}$ by setting up, and solving, a slightly different system of congruences. Thus, we would like to find integers $p,q$ for which $[v_1] = px + q[v_n]$; i.e., integers $p,q$ for which 
$$ p\frac{F(n)}{d(n)}[v_n]+p\frac{F(n-1)-1}{d(n)}[v_1]+q[v_n]=[v_1].$$
This time, instead of invoking Proposition \ref{sizeofK0isH(n)}, as we did in the other case, we will actually need something stronger, concretely Proposition \ref{Step1}, so that the previous equation will have a solution in $(p,q)$ if (perhaps not only if) the following system of congruences has a solution:
\begin{equation} \tag{$\ddagger$}
\begin{cases}
p\frac{F(n)}{d(n)} +q & \equiv 0\m \frac{H_2(n)}{d(n)}   \\
p\frac{F(n-1)-1}{d(n)} & \equiv 1 \m \frac{H_2(n)}{d(n)} 
\end{cases}
\end{equation}
And this system has a solution if and only if $\gcd(\frac{F(n-1)-1}{d(n)},\tfrac{H_2(n)}{d(n)})=1$, which is what the following result will prove for the cases of $n$ that we need (and for some additional cases as a by-product). Specifically, even though only the $n\equiv 0\m 6$ case remains, since the formula for $d(n)$ in those cases is given in terms of mod $4$, we will have to consider all even natural numbers  $n$, and then contemplate the two cases $n\equiv 0\text{ or }2\m 4$ separately.

\begin{proposition}\label{SecondFormulaForStep2} Let $n\in 2\N$, then: 
$$\gcd\left(\frac{F(n-1)-1}{d(n)},\frac{H_2(n)}{d(n)}\right)=1$$
\end{proposition}
\begin{proof}
  For notational simplicity we denote $\gcd\left(\frac{F(n-1)-1}{d(n)},\frac{H_2(n)}{d(n)}\right)$ by $B(n)$. The proof will again distinguish cases.
\vskip0.2cm

\underline{Case 1:} $n\equiv 0\m 4$. 

Write $n=2m$ for even $m$. In this situation, by Case 4 of Proposition \ref{Step1} we have: $$F(n-1)-1=d(n)(F(\tfrac{n}{2})+F(\tfrac{n}{2}-2))=d(n)(F(m)+F(m-2))$$ and also Proposition \ref{d(n)formulas} gives $d(n)=F(m)$. Thus
\begin{align*}
H_2(n) = & H_2(2m)=F(2m+1)+F(2m-1)-1-(-1)^{2m}\\
       = & (F(m+1)^2+F(m)^2)+(F(m)^2+F(m-1)^2)-2\\
       = & (F(m)+F(m-1))^2+2F(m)^2+F(m-1)^2-2\\
       = & 3F(m)^2+2F(m)F(m-1)+2(F(m-1)^2-1)=F(m)(3F(m)+2F(m-1)+2F(m-2))
\end{align*}
Therefore we can simplify as follows:
\begin{align*}
B(n)  = & \g(\tfrac{F(n-1)-1}{d(n)},\tfrac{H_2(n)}{d(n)})=\g(F(m)+F(m-2),F(m)+2F(m-1)+2(F(m)+F(m-2)))\\
    \eg & \g(F(m)+F(m-2),F(m+1)+F(m-1))\\
      = & \g(F(m)+F(m-2),(F(m)+F(m-2))+F(m-1)+F(m-3))\\
    \eg & \g(F(m)+F(m-2),F(m-1)+F(m-3))\\
      = & \g(F(m-1)+F(m-3),F(m)+F(m-2))=\dots=\g(F(4)+F(2),F(3)+F(1))=1
\end{align*}
\vskip0.2cm

\underline{Case 2:} $n\equiv 2\m 4$. 

Write $n=2m$ for odd $m$. Using some of the formulas we got in Case 3 of Proposition \ref{Step1} we have: $d(n)=F(m+1)+F(m-1)$ and $F(n-1)-1=d(n)F(m-1)$. We express $H_2(n)$ in terms of Fibonacci numbers prior to tackling the greatest common divisor:
\begin{align*}
H_2(n) = & F(2m+1)+F(2m-1)-1-(-1)^{2m} \hskip1cm {\rm by \ (HtoF)}\\
 = &(F(m+1)^2+F(m)^2)-1+(F(n-1)-1)\\
       = & F(m+1)^2+F(m+1)F(m-1)+(F(n-1)-1)\\
       = & F(m+1)(F(m+1)+F(m-1))+(F(n-1)-1) \\
       = & F(m+1)d(n)+(F(n-1)-1)
\end{align*}
Then the expression that we are after can be simplified in the following manner:
\begin{align*}
B(n) = & \g(\tfrac{F(n-1)-1}{d(n)},\tfrac{F(n-1)-1}{d(n)}+\tfrac{F(m+1)d(n))}{d(n)}) \eg \g(F(m-1),F(m+1))\\
     = & \g(F(m-1),F(m)+F(m-1))\eg \g(F(m-1),F(m))\\
     = & \g(F(m-1),F(m-1)+F(m-2)) \eg \g(F(m-1),F(m-2))\\
     = & \g(F(m-2),F(m-1))=\dots=\g(F(1),F(2))=1
\end{align*}
This completes the proof.
\end{proof}

We have now all the ingredients in hand to prove the main result of this article.

\begin{theorem}\label{TheTheorem} Let $n\in \N$. Then $$K_0(L_K(C_n^2))\cong \Z_{d(n)}\times \Z_{\frac{H_2(n)}{d(n)}}.$$ Furthermore, $K_0(L_K(C_n^2))$ is cyclic if and only if $n=2$, $n=4$,  $n\equiv 1\m 6$, or $n\equiv 5\m 6$. 
\end{theorem}
\begin{proof}
The statement on the isomorphism of groups is what we have done with Step 1 and Step 2 in the previous discussion by applying Lemma \ref{grouplemma} to the appropriate elements $x$ and $y$.

The final sentence hold as follows. On the one hand we know by Proposition \ref{d^2|H_2} that $d(n)^2$ divides $H_2(n)$, so that $\frac{H_2(n)}{d(n)}$ always has $d(n)$ as a factor. Hence, the only way the group product can be cyclic is when $d(n)=1$ (so that we only have the factor $\Z_{H_2(n)}$). But the formula of $d(n)$ given in Proposition \ref{d(n)formulas}, together with the fact that the Fibonacci sequence is increasing, immediately yield that the only possible values of $n$ are those given in the statement.
\end{proof}

As an easy remark, and as happened in the case of $C_n^1$, the algebras $\{L_K(C_n^2)\ |\ n\in \N\}$ are mutually non-isomorphic as their $K_0$ groups have size $H_2(n)$, and this is an strictly increasing sequence.

We are also in position to apply the Algebraic KP Theorem to explicitly realize the algebras $L_K(C_n^2)$ as the Leavitt path algebras of graphs having three vertices (we will see that in the most interesting cases it is not possible to realize them with fewer vertices than that). We got the inspiration for the graph $E_n$ below from \cite[Proposition 3.6]{EKTW}.

\begin{proposition} Let $n\in \N$. The algebra $L_K(C_n^2)$ is isomorphic to the Leavitt path algebra of the graph $E_n$ with three vertices given by
$$
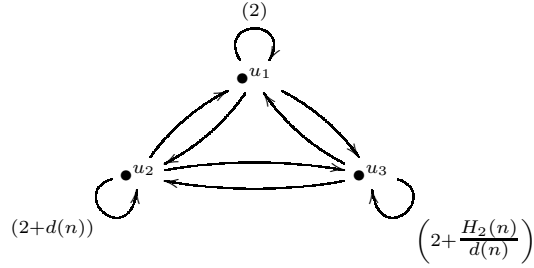
\xymatrix{ &   {\bullet}^{u_1}  \ar@(lu,ru)^{(2)} \ar@/^{5pt}/[dl]\ar@/^{5pt}/[dr] &  \\
           {\bullet}^{u_2} \ar@/^{5pt}/[ur] \ar@/^{5pt}/[rr] \ar@(l,d)_{(2+d(n))}&   & {\bullet}^{u_3} \ar@/^{5pt}/[ll] \ar@/^{5pt}/[ul] \ar@(r,d)^{\left(2+\tfrac{H_2(n)}{d(n)}\right)}
}$$ where the numbers in parentheses indicate the number of loops at a given vertex. Furthermore, when $K_0(L_K(C_n^2))$ is not cyclic, this realization is minimal in the sense that it is not possible to realize $L_K(C_n^2)$ up to isomorphism as the Leavitt path algebra of a graph having less than three vertices.
\end{proposition} 
\begin{proof}
Clearly the graph $E_n$ satisfies the conditions for $L_K(E_n)$ to be unital purely infinite simple.    The incidence matrix of $E_n$ is $A_{E_n}=\left(\begin{matrix} 2 & 1 & 1Ê\\ 1 & 2+d(n) & 1Ê\\ 1 & 1 & 2+\tfrac{H_2(n)}{d(n)} \end{matrix}\right),$ so that $$I_n-A^t_{E_n}=-\left(\begin{matrix} 1 & 1 & 1Ê\\ 1 & 1+d(n) & 1Ê\\ 1 & 1 & 1+\tfrac{H_2(n)}{d(n)} \end{matrix}\right).$$ 
\vskip-.2cm
\noindent
An easy computation shows that the Smith normal form of $I_n-A^t_{E_n}$ is just $\left(\begin{matrix} 1 & 0 & 0Ê\\ 0 & d(n) & 0Ê\\ 0 & 0 & \tfrac{H_2(n)}{d(n)} \end{matrix}\right),$ which immediately yields that  $K_0(L_K(E_n))$ is isomorphic to $\Z_{d(n)}\times \Z_{\frac{H_2(n)}{d(n)}}$. 

Also, it is straightforward to check that $\det(I_n-A^t_{E_n})=-H_2(n)<0$.

Finally,  by invoking the relation in $K_0(L_K(E_n))$ at  $u_1$, we have  $$[u_1] + [u_2] + [u_3] = (2[u_1] + [u_2] + [u_3]) + [u_2] + [u_3] = 2([u_1] + [u_2] + [u_3]),$$ so that $\sigma = [u_1] + [u_2] + [u_3]$ satisfies $\sigma =2\sigma$ in the group $K_0(E_n)$, so that $\sigma = [u_1] + [u_2] + [u_3]$ is the identity element of $K_0(E_n)$.

Thus  the purely infinite simple unital Leavitt path algebras  $L_K(C_n^2)$ and $L_K(E_n)$ have these properties:  $K_0(L_K(C_n^2)) \cong K_0(L_K(E_n))$ (as each is isomorphic to $\Z_{d(n)}\times \Z_{\frac{H_2(n)}{d(n)}}$); this  isomorphism takes $[L_K(C_n^2)]$ to $[L_K(E_n)]$ (as each of these is the identity element in their respective $K_0$ groups);  and  both  $\det(I_n-A^t_{C_n^2})$   and   $\det(I_n-A^t_{E_n})$ are negative.  Thus  the graphs $C_n^2$ and $E_n$ satisfy the   hypotheses of  the Algebraic KP Theorem, and so the  desired isomorphism $L_K(C_n^2) \cong L_K(E_n)$ follows.

Now suppose that $E$ is any graph for which $L_K(E)$ is purely infinite simple.  If $E$ has just one vertex then clearly $K_0(L_K(E))$ is cyclic.   But the same conclusion is true as well in case $E$ has just two vertices (say, $v_1$ and $v_2$),   and $[L_K(E)] = 0$ in $K_0(L_K(E))$:  for in this situation we have $[v_1] + [v_2] = 0$, which gives that $[v_2] \in \langle [v_1] \rangle$, and thus $[v_1]$ is a generator of $K_0(L_K(E))$.   Therefore, when 
 $K_0(L_K(C_n^2))$ is not cyclic,
 an application of Proposition \ref{sumverticesisidentity} finishes the result.
\end{proof}

\section*{Acknowledgments}

The authors are grateful to Attila Egri-Nagy, who brought to the attention of the authors the potential connection between Leavitt path algebras and Cayley graphs during the conference ``Graph C*-algebras, Leavitt path algebras and symbolic dynamics", held at the University of Western Sydney, February 2013.


\bibliographystyle{amsalpha}

\end{document}